\documentclass[a4paper,12pt]{amsart}

\usepackage{amsmath}
\usepackage{amssymb}
\usepackage{graphicx}
\usepackage{cases}
\usepackage{cite}

\calclayout

\newtheorem{theorem}{Theorem}[section]

\newtheorem{lemma}[theorem]{Lemma}
\newtheorem{proposition}[theorem]{Proposition}
\theoremstyle{definition}

\newtheorem{example}[theorem]{Example}
\newtheorem{remark}[theorem]{Remark}

\numberwithin{equation}{section}

\DeclareMathOperator{\Fix}{Fix}

\newcommand{\HH}{\mathcal{H}}
\newcommand{\NN}{\mathbb{N}}
\newcommand{\RR}{\mathbb{R}}

\title[Projection Algorithms for Finite Sum]{Projection Algorithms for Finite Sum Constrained Optimization}%

\author[H. K. Xu]{Hong-Kun Xu}
\address[H. K. Xu]{School of Science, Hangzhou Dianzi University,
     Hangzhou 310018, China}
\email{\tt xuhk@hdu.edu.cn}

\author[V. Roshchina]{Vera Roshchina$^*$}\thanks{$^*$Corresponding author.}
\address[V. Roshchina]{School of Mathematics and Statistics, University of New South Wales,
     Sydney, NSW 2052, Australia}
\email{\tt v.roshchina@unsw.edu.au}



\date{}

\begin{document}

\begin{abstract}

Parallel and cyclic projection algorithms are proposed for
minimizing the sum of a finite family of convex functions over the intersection
of a finite family of closed convex subsets of a Hilbert space. These algorithms are of predictor-corrector type, with each main iteration consisting of an inner cycle of
subgradient descent process followed by a projection step. We prove the convergence of these methods to an optimal
solution of the composite minimization problem under investigation upon assuming
boundedness of the gradients at the iterates of the local functions and
the stepsizes being chosen appropriately, in the finite-dimensional setting. We also discuss generalizations and  limitations of the proposed algorithms and our techniques.
\end{abstract}

\maketitle

\section{Introduction}

We are concerned with a composite minimization problem, that is, we consider the case
where the objective function is decomposed into the sum of a finite family of convex
functions and the set of constraints is the intersection of finitely many closed convex
subsets of a real Hilbert space $\HH$. Precisely, the minimization problem under investigation in this paper
is of the form
\begin{equation}\label{min:f}
\min_{x\in C:=\bigcap_{i=1}^M C_i}\, f(x):=\sum_{j=1}^N f_j(x),
\end{equation}
where $M, N$ are positive integers, each set $C_i$ is a nonempty closed convex subset of $\HH $,
and each component function $f_j: \HH  \to \mathbb{R}$ is a convex function.
We always assume the feasible set $C\not=\emptyset$.

Large-scale optimization problems of form (\ref{min:f}) naturally arise in modern applications, in particular, network design \cite{LCJ2013,FG2014} and machine learning \cite{HTF2009,SNW2011,KH2004}. When the constraint of~\eqref{min:f} is defined explicitly by the system of inequalities, penalty and augmented Lagrangian  techniques, as well as proximal and bundle methods can be applied to this problem. 
However, when projections onto the constraint sets are readily available, the treatment of constraints via projections techniques may be preferable as computationally robust and memory efficient. One approach that allows to apply projection methods to \eqref{min:f} is to replace the optimization problem~\eqref{min:f} with a sequence of CFPs as is done in \cite{Gibali2018}. Our development is more direct: we build on the ideas of \cite{DN2009} to prove the convergence of subgradient projection techniques that utilize projections onto individual constraint sets.  We note that despite a large body of work dedicated to solving convex feasibility problems via projection methods (see  \cite{RZ2017,BLT2017,LS2018,Bau,CRZ} for recent advancements and \cite{BC,Cegielski} for textbook exposition)  and vast literature on optimization methods that utilize a single projection onto the constraint set  (for recent works see, e.g. \cite{LD2018,ST2018,Necoara2018}), little is done in  combining optimization and projection steps on several sets, beyond the aforementioned paper by De~Pierro and Helou Neto~\cite{DN2009}. Our aim is to make a substantial contribution towards bridging this gap. Recent progress on forcing the convergence of Douglas--Rachford type methods to the smallest norm feasible point \cite{AC2018} also indicates that it may be possible to extend our approach to a larger class of projection techniques.

The convex feasibility problem (CFP) \cite{BB1996,Combettes1997} is formulated as
\begin{align}\label{cfp:1}
\mbox{ finding a point $x^*$ with the property:}\  x^*\in \bigcap_{i=1}^N C_i.
\end{align}
Thus, the composite minimization problem (\ref{min:f}) can alternatively be rephrased as
finding a solution to the convex feasibility problem (\ref{cfp:1}) which also minimizes the composite function $f$
as defined in (\ref{min:f}). Consequently, two points should be taken into consideration of
 algorithmic approaches to (\ref{min:f}):
 \begin{itemize}
\item[(a)] the descent property of the values of the
objective function $f$, and
\item[(b)] the (approximate) feasibility of the iterates generated by the algorithm.
\end{itemize}
To illustrate these points we consider the special case where $M=N=1$ and the function $f_1$ is smooth.
In this case, (\ref{min:f}) is reduced to the
constrained convex minimization:
\begin{equation}\label{min:1:f1}
\min_{x\in C_1}  f_1(x).
\end{equation}
The gradient-projection algorithm (GPA) can solve (\ref{min:1:f1}): GPA generates a sequence $\{x_k\}$ by
the recursion process:
\begin{equation}\label{al:gpa}
x_{k+1}=P_{C_1}(x_k-\lambda_k\nabla f_1(x_k)),
\end{equation}
where the initial guess $x_0\in \HH $ is chosen arbitrarily, and $\lambda_k>0$ is the stepsize.
Assume:
\begin{enumerate}
\item[(A1)] The gradient of $f_1$, $\nabla f_1$, is $\alpha$-Lipschitz (for some $\alpha\ge 0$):
$$\|\nabla f_1(x)-\nabla f_1(z)\|\le \alpha\|x-z\|,\quad x,z\in \HH ;$$
\item[(A2)] The sequence of stepsizes, $\{\lambda_k\}$, satisfies the condition:
$$0<\liminf_{k\to\infty}\lambda_k\le\limsup_{k\to\infty}\lambda_k<\frac{2}{\alpha}.$$
\end{enumerate}
It is then easy to find that both points (a) and (b) hold (actually, (b) holds trivially);
moreover, the sequence $\{x_k\}$ generated by GPA (\ref{al:gpa}) converges \cite{Polyak1987,Xu2011}
weakly to a solution of (\ref{min:1:f1}) (if any).

Observe that the splitting of the objective function $f$ into the
sum of $N$ (simpler) component functions, and the set $C$ of constraints into the intersection of $M$ (simpler) convex subsets
aims at providing more efficient algorithmic approaches to (\ref{min:f}) by utilizing the simpler structures of
the component functions $\{f_j\}$ (for instance, the proximal mappings of $f_j$ are computable \cite{CW2005})
and of the sets $\{C_i\}$ (for instance, the projections $P_{C_i}$ possess closed formulae).  This means that when we study algorithms for the composite optimization problem (\ref{min:f}),
we should use individual component functions and individual subsets at each iteration, not the full sum of the
component functions $\{f_j\}$, nor the full intersection of the sets $\{C_i\}$.

The purpose of this paper is to analyse the convergence of parallel and cyclic
projection algorithms for solving the optimization problem \eqref{min:f}, significantly expanding the results of De Pierro and Helou Neto in \cite{DN2009} who focussed on the sequential projections version of the method. We provide a unified analysis of all three methods in the finite-dimensional setting.

The projection algorithms studied in this paper start with an arbitrary point $x_0\in \HH $ and produce the iterates $x_{k+1}$ ($k\ge 0$), alternating between  subgradient  and  projection steps.

The generic form of our projection algorithm is as follows.
\begin{equation}
\begin{cases}\tag{Projection algorithm}
x_{k,0}=x_k, \\
x_{k,j}=x_{k,j-1}-\lambda_k v_{k,j}, \; v_{k,j} \in \partial f_{j}(x_{k,j-1}) ,\;  j = 1,2,\cdots, N,\\
x_{k+1}=V_{k+1}(x_{k,N}),
\end{cases}
\end{equation}
Here by $\partial f_{j}(x)$ we denote the Moreau-Rockafellar subdifferential of the convex function $f_j$ at a point $x$ for any $j\in \{1,\dots, N\}$, and  $V_{k+1}$ is the (modification of) projection operator, distinguishing the three methods. We have explicitly for $k\geq 0$
$$
V_{k+1}: = \begin{cases}
P_{C_M}\cdots P_{C_1}, & \text{ for sequential projections};\\
 P_{C_{[k+1]}},\quad  [k+1] =(k\mod M)+1, & \text{ for cyclic projections};\\
\sum_{i=1}^M \beta_iP_{C_i},\quad  \beta_i>0 \; \forall i,\; \sum_{i=1}^M\beta_i=1, & \text{ for parallel projections}.
\end{cases}
$$

The \emph{sequential projection} algorithm was introduced by De Pierro and Helou Neto in \cite{DN2009},
in this case the projection step is a full cycle of projections onto the $M$ sets whose intersection comprises the feasibility region. Explicitly we have
\begin{equation}
\begin{cases}\tag{Sequential projections}
x_{k,0}=x_k, \\
x_{k,j}=x_{k,j-1}-\lambda_k v_{k,j}, \; v_{k,j} \in \partial f_{j}(x_{k,j-1}) ,\;  j = 1,2,\cdots, N,\\
x_{k+1}=P_{C_M}\cdots P_{C_1}x_{k,N}.
\end{cases}
\end{equation}
In the finite-dimensional case, De Pierro and Helou Neto discussed the convergence properties of the above algorithm (note that generalized the original method slightly, replacing gradients with subgradients; this does not affect the convergence analysis that relies on the convexity of the component objective functions rather than their differentiability).
Moreover, they raised several open questions regarding projection algorithms for solving (\ref{min:f}),
one of which is whether the sequential projections in their algorithm can be
replaced with the parallel projections. We answer this question in the affirmative, not only for parallel, but also for the cyclic version of the algorithm.

Our main result is the following direct generalization of \cite[Theorem~1]{DN2009}.

\begin{theorem}\label{th:main} Let $\dim \HH <\infty$, suppose that the sets $C_1,\dots, C_M\subset \HH $ are closed and convex, and let $x_0\in \HH $. Assume that the real-valued convex functions $f_1$, \dots, $f_N$  are defined on some convex subsets $D_1$, \dots, $D_N$ of $\HH $ such that $x_{k,{j-1}}\in D_j$, $j \in \{1,\dots, N\}$, $k\geq 0$ (for a choice of cyclic, sequential or parallel projection algorithm) and there exist constants $L_1$, \dots, $L_N$ such that
	$$
\max_{v\in \partial f_j(x_{k,j-1})}\left\|v\right\|\le L_j, \quad j=1,2,\cdots, N, \quad k\ge 0.
$$	
Moreover, assume that the sequence $(x_k)$ (obtained via the chosen method) is bounded and
$$
0<\lambda_k\to 0\quad{\rm and}\quad \sum_{k=0}^\infty\lambda_k=\infty.
$$
Then the sequence $\{f(x_k)\}$ converges to the optimal value $f^*:=\inf_{x\in C} f(x)$, and every cluster point of $\{x_k\}$
is an optimal solution of (\ref{min:f}), given that the solution set is nonempty.
\end{theorem}

Note that our assumptions are standard in the analysis of numerical methods, and can be replaced by more constructive or convenient conditions, with some loss of generality.

The proof of our main result (Theorem~\ref{th:main}) relies on the key property of asymptotic feasibility (that ensures the cluster points of the iterative sequence converge to the feasible set). We prove asymptotic feasibility for the methods of parallel and cyclic projections in Section~\ref{sec:4}, and present the complete proof of Theorem~\ref{th:main} in Section~\ref{sec:mainpf}. Note that even though we follow the general framework of De Pierro and Helou Neto, our proofs of asymptotic feasibility for cyclic and parallel projections are based on entirely different ideas.

We begin our discussion with introducing some notation and other preliminary information and results in Section~\ref{sec:discussion}, and after presenting the proof of the main results in Sections~\ref{sec:4} and~\ref{sec:mainpf}, provide a discussion of some generalizations including the infinite-dimensional setting, and some practical improvements and modifications of the methods.

\section{Notation and Preliminaries}
\label{sec:discussion}

The fundamental tool of our argument in this paper is the concept of projections.
Let $\HH $ be a real Hilbert space with inner product $\langle\cdot,\cdot\rangle$
and norm $\|\cdot\|$, respectively, and let $C$ be a nonempty closed convex subset of $\HH $.
The (nearest point) projection from $\HH $ onto $C$, dented by $P_C$, is defined by
\begin{equation}\label{eq:proj}
P_Cx:=\arg\min_{y\in C}\|x-y\|,\quad x\in \HH .
\end{equation}
The following well-known properties are pertinent to our argument in Section 3.

\begin{proposition}\label{prop:proj} Let $\HH$ be a real Hilbert space, and for any closed convex set $C\subseteq \HH$ let $P_C$ be the projector operator defined by \eqref{eq:proj}. Then the following properties hold.
\begin{itemize}
\item[{\rm (i)}]
 $\langle x-P_{C}x,y-P_{C}x\rangle\le 0$ for all  $x\in \HH $ and $y\in C$.
\item[{\rm (ii)}]
 $\langle P_{C}x-P_{C}y,x-y\rangle\ge \|P_Cx-P_Cy\|^2$ for all $x, y\in \HH $;
in particular, $P_C$ is nonexpansive, namely,
$$\| P_Cx-P_Cy\|\le\|x-y\|,\quad x,y\in \HH .$$
\item[{\rm (iii)}]
$\|P_{C}x-y\|^2\le\|x-y\|^2-\|P_{C}x-x\|^2$ for all $x\in \HH $ and $y\in C$.
\end{itemize}
\end{proposition}

We also define the distance function from a point $x\in \HH$ to a set $E\subseteq \HH$ as
$$
d_E(x):=\inf\{\|x-y\|: y\in E\}.
$$
Observe that for a closed convex set $C$ we have $d_C(x) = \|x-P_Cx\|$.

As mentioned earlier, the CFP (\ref{cfp:1}) can be solved by the projection onto convex sets method (POCS), whose convergence is well-understood in the general context of real Hilbert spaces. We recall the well-known convergence results of two major POCS algorithms \cite{BB1996,Combettes1997,OPX2006,Wang-Xu2011}.

\begin{theorem}\label{th:cfp}
	Beginning with an arbitrarily chosen initial guess $x_0\in \HH $, we iterate $\{x_k\}$ in either one of the
	following two projection algorithms:
	\begin{enumerate}
		\item[{\rm (i)}] Sequential (cyclic) projections: $x_{k+1}=P_{C_M}\cdots P_{C_1}x_k$;
		\item[{\rm (ii)}] Parallel projections: $x_{k+1}=\sum_{j=1}^M\beta_j P_{C_j}x_k$, with $\beta_j>0$ for all $j$ and
		$\sum_{j=1}^M\beta_j=1$;
	\end{enumerate}
	Then $\{x_k\}$ converges weakly to a solution of CFP (\ref{cfp:1}), given that this solution set is nonempty.
\end{theorem}

Another key notion in our discussion is that of a convex function and Moreau--Rockafellar subdifferential \cite{BR1965}.
Let $D$ be a convex subset of $\HH$, and let $f:D\to \RR$ be a convex function. A subgradient of $f$ at $x\in D$ is a vector $v\in \HH$ such that
$$
f(y)\geq f(x)+ \langle y-x,v\rangle \quad \forall y\in D.
$$
The set of all subgradients of $f$ at $x$ is called the subdifferential and is denoted by $\partial f(x)$.

Let
$$S^*:=\left\{x^*\in C: f(x^*)=\inf_{x\in C} f(x)\right\}\quad {\rm and}\quad
f^*:=\inf_{x\in C} f(x)$$
be the set of optimal solutions and the optimal value of the composite minimization problem (\ref{min:f}),
respectively. We shall always assume from now and onwards that $S^*\not=\emptyset$.

Two problems are pertinent:
\begin{enumerate}
	\item[(a)] The sequence $\{x_k\}$ would (weakly) converge to an optimal solution $x^*\in S^*$;
	\item[(b)] The sequence $\{f(x_k)\}$ would converge to the optimal value $f^*$.
\end{enumerate}
If the answer to (a) is affirmative, then the answer to (b) is also positive.

The assumptions of Theorem~\ref{th:main}  play a key role in establishing the aforementioned properties. We state and discuss these assumptions here explicitly for the clarity of  exposition.

First, we make a standard assumption on the divergence of the series of diminishing stepsizes used at the gradient cycle of our projection algorithm: we require that
\begin{equation}\label{eq:lambda1}
0<\lambda_k\to 0\quad{\rm and}\quad \sum_{k=0}^\infty\lambda_k=\infty.
\end{equation}
The first condition ensures that the steps we make are indeed descent steps, and that the gradient step does not derail our progress with the convergence of projection steps to the feasible set. The second condition ensures that there is no artificial restriction on how far can the sequence of iterates depart from the initial point.

The second key assumption is a uniform Lipschitz bound on the components of the objective function. Explicitly, we use the following assumption on the subgradients of our functions,
\begin{equation}\label{eq:Lip}
\max_{v\in \partial f_j(x_{k,j-1})}\left\|v\right\|\le L_j, \quad j=1,2,\cdots, N, \quad k\ge 0,
\end{equation}
and we also let $L:= \sum_{j=1}^N L_j$.
Observe that this condition is satisfied naturally when these (real-valued) functions are defined on the whole finite-dimensional  space $\HH $ and the sequence $(x_k)$ is bounded.  It is also well-known (see  \cite[Proposition~7.8]{BB1996}) that the condition of a function having bounded gradients (subdifferentials) on bounded sets is equivalent to the function being bounded on bounded sets in the finite-dimensional setting.

\section{Asymptotic Feasibility of Parallel and Cyclic Projections}
\label{sec:4}

We are ready to prove two major technical results that concern the asymptotic feasibility of parallel and cyclic projections (Lemmas~\ref{le:ppa2} and \ref{lem:conv} respectively). Note that the relevant statement for the sequential projections was shown in \cite{DN2009}.

\subsection{Asymptotic Feasibility for Parallel Projections}

Recall that the \emph{parallel projection algorithm} (PPA) utilizes a convex combination of the projections on the sets $C_1$,\dots, 
$C_M$ on its projection step:
\begin{equation}
\begin{cases}\tag{PPA}
x_{k,0}=x_k, \\
x_{k,j}=x_{k,j-1}-\lambda_k v_{k,j}, \quad v_{k,j} \in \partial f_{j}(x_{k,j-1}) ,\quad  j = 1,2,\cdots, N,\\
x_{k+1}=\sum_{i=1}^M \beta_iP_{C_i}x_{k,N},\quad  \beta_i>0 \; \forall i,\; \sum_{i=1}^M\beta_i=1.
\end{cases}
\end{equation}

Our goal is to prove the following result. We begin with several technical claims that we use in the proof that is deferred to the end of this subsection.

\begin{lemma}\label{le:ppa2}
	Assume ${\rm dim}\, \HH <\infty$, (\ref{eq:Lip}), and $\lambda_k\downarrow 0$, and that the sequence $\{x_k\}$ generated by the method of parallel projections is bounded.
	Then $\{x_k\}$ is asymptotically feasible, that is, $\lim_{k\to\infty} d_C(x_k)=0$.
\end{lemma}

The following technical result is used in the subsequent proofs.

\begin{lemma}\label{le:prop-ppa}
Let $\{x_k\}$ be a sequence generated by the parallel projections algorithm and assume that the Lipschitz condition (\ref{eq:Lip}) is satisfied. Then
\begin{enumerate}
\item[(i)]
$\|x_{k,N}-x_k\|\le L\lambda_k$, where $L=\sum_{j=1}^N L_j$.
\item[(ii)]
$\|x_{k+1}-z\|^2\le \|x_{k,N}-z\|^2-\sum_{j=1}^M\beta_jd_{C_j}^2(x_{k,N})$ for $z\in C$.
\item[(iii)]
$d_{C}^2(x_{k+1})\le d_{C}^2(x_{k,N})-\sum_{j=1}^M\beta_jd_{C_j}^2(x_{k,N}).$
\item[(iv)]
$d_{C}^2(x_{k+1})\le d_{C}^2(x_{k})-\sum_{j=1}^M\beta_jd_{C_j}^2(x_{k})+2\lambda_k(2d_C(x_k)+\lambda_kL)$.
\end{enumerate}

\end{lemma}
\begin{proof}
(i) We have
\begin{equation*}
\|x_{k,N}-x_k\|
\le\sum_{j=1}^N \|x_{k,j}-x_{k,j-1}\|=\sum_{j=1}^N \lambda_k\|v_{j,k}\|
\le \sum_{j=1}^N \lambda_kL_j=L\lambda_k.
\end{equation*}

(ii)
For $z\in C$, we have
\begin{align*}
\|x_{k+1}-z\|^2&=\left\|\sum_{j=1}^M\beta_jP_{C_j}x_{k,N}-z\right\|^2\nonumber \\
&\le\sum_{j=1}^M\beta_j\|P_{C_j}x_{k,N}-z\|^2 \quad {\rm by\ convexity \ of}\  \|\cdot\|^2 \nonumber\\
&\le\sum_{j=1}^M\beta_j(\|x_{k,N}-z\|^2-\|x_{k,N}-P_{C_j}x_{k,N}\|^2) \quad \text{(by Proposition~\ref{prop:proj}(iii))}\nonumber\\
&=\|x_{k,N}-z\|^2-\sum_{j=1}^M\beta_jd_{C_j}^2(x_{k,N}).
\end{align*}

(iii)
This is a straightforward consequence of (ii).

(iv)
This is easily derived from (iii), (i) and the fact that a distance function of a convex set is
Lipschitz continuous with Lipschitz constant one:
$$|d_K(x)-d_K(y)|\le\|x-y\|.$$
\end{proof}

\begin{lemma}\label{le:ppa1}
Assume ${\rm dim}\, \HH <\infty$, $\lambda_k\downarrow 0$, the condition (\ref{eq:Lip}) is satisfied, and $\{x_k\}$ is bounded.
Then for any $\varepsilon>0$, there exists $\delta>0$ such that
\begin{equation}\label{d-C}
d_C^2(x_{k+1})\le d_C^2(x_k)-\delta
\end{equation}
 whenever $k$ is such that $d_C(x_k)\ge\varepsilon$.
Consequently, $\liminf_{k\to\infty}d_C(x_k)=0$.
\end{lemma}
\begin{proof}
Suppose not; then for some $\varepsilon_0>0$, we have a subsequence $\{x_{k_l}\}$ of $\{x_k\}$
such that $d_C(x_{k_l})\ge\varepsilon_0$ and
\begin{equation}
d_C^2(x_{k_l+1})> d_C^2(x_{k_l})-\frac1{l}
\end{equation}
for all $l\ge 1$. It turns out from Lemma \ref{le:prop-ppa}(iv) that
\begin{equation}\label{sum-j}
\sum_{j=1}^M\beta_j d_{C_j}^2(x_{k_l})\le d_C^2(x_{k_l})-d_C^2(x_{k_l+1})+O(\lambda_{k_l})<\frac1{l}+O(\lambda_{k_l})\to 0\
({\rm as}\   l\to\infty).
\end{equation}
Since $\{x_k\}$ is a bounded sequence in a finite dimensional space, we may assume that $x_{k_l}\to\hat{x}$.
We then get
\begin{equation}\label{s}
\sum_{j=1}^M\beta_j d_{C_j}^2(\hat{x})=0.
\end{equation}
This implies that $\hat{x}\in C_j$ for every $j$; hence, $\hat{x}\in C$.
This contradicts the fact that $d_C(\hat{x})=\lim_{l\to\infty}d_C(x_{k_l})\ge\varepsilon_0>0$.
\end{proof}

We are now ready to prove Lemma~\ref{le:ppa2}.

\begin{proof}[Proof of Lemma~\ref{le:ppa2}]
By Lemma~\ref{le:ppa1} we have $\liminf_{k\to\infty} d_C(x_k)=0$, hence, we can take
$k_0$ such that $d_C(x_{k_0})<\varepsilon$ and
$\lambda_{k}L<\frac12\varepsilon$ for all $k\ge k_0$.
Let $k\ge k_0$. Consider two cases.

Case 1: $d_C(x_k)<\varepsilon$. In this case, we have by Lemma \ref{le:prop-ppa}(iii)
\begin{equation}\label{d-Cx}
d_C(x_{k+1})\le d_C(x_{k,N})\le d_C(x_k)+\|x_k-x_{k,N}\|\le d_C(x_k)+\lambda_kL<\frac32\varepsilon.
\end{equation}

Case 2: $d_C(x_k)\ge\varepsilon$. Using Lemma \ref{le:ppa1}, we obtain
$d_C(x_{k+1})<d_C(x_{k})$.

We now prove, for all $i\ge 0$,
\begin{equation}\label{dCx}
d_C(x_{k_0+i})<\frac32\varepsilon.
\end{equation}
Indeed, (\ref{dCx}) is trivial when $i=0$. Assume (\ref{dCx}) holds for $i$.
If $d_C(x_{k_0+i})\ge\varepsilon$, then $d_C(x_{k_0+i+1})<d_C(x_{k_0+i})<\frac32\varepsilon$;
if $d_C(x_{k_0+i})<\varepsilon$, then, by Case 1, we get $d_C(x_{k_0+i+1})<\frac32\varepsilon$.
Hence, (\ref{dCx}) also holds for $i+1$.

Now it turns out from (\ref{dCx}) that
$\limsup_{k\to\infty} d_C(x_k)\le\frac32\varepsilon$, and Lemma~\ref{le:ppa2} is proven.
\end{proof}

Note that Lemmas~\ref{le:ppa1} and \ref{le:ppa2} can be generalized for the infinite-dimensional setting. We discuss this in more detail in Section~\ref{ssec:inf}.

\begin{remark}
We include a version of Lemma \ref{le:prop-ppa} for the sequential projection algorithm (SPA) that generates a sequence $\{x_k\}$ 
via the following iteration process: 
\begin{equation}
\begin{cases}\tag{SPA}
x_{k,0}=x_k, \\
x_{k,j}=x_{k,j-1}-\lambda_k v_{k,j}, \  v_{k,j} \in \partial f_{j}(x_{k,j-1}),\   j = 1,2,\cdots, N,\\
x_{k+1}=P_{C_M}\cdots P_{C_1}x_{k,N}.
\end{cases}
\end{equation}
\begin{lemma}\label{le:prop-spa}
Let $\{x_k\}$ be generated by (SPA) and assume
that the Lipschitz condition (\ref{eq:Lip}) is satisfied. Then
\begin{enumerate}
\item[(i)]
$\|x_{k,N}-x_k\|\le L\lambda_k$, where $L=\sum_{j=1}^N L_j$.
\item[(ii)]
$\|x_{k+1}-z\|^2\le \|x_{k,N}-z\|^2-\sum_{j=1}^M d_{C_j}^2(P_{C_{j-1}}\cdots P_{C_1}x_{k,N})$ for $z\in C$.
\item[(iii)]
$d_{C}^2(x_{k+1})\le d_{C}^2(x_{k,N})-\sum_{j=1}^M d_{C_j}^2(P_{C_{j-1}}\cdots P_{C_1}x_{k,N})$.
\item[(iv)]
$d_{C}^2(x_{k+1})\le d_{C}^2(x_{k})-\sum_{j=1}^M d_{C_j}^2(P_{C_{j-1}}\cdots P_{C_1}x_{k})+\mu_k$.
\end{enumerate}
Here $\mu_k:=2\lambda_k L(\lambda_k L+d_C(x_k)+\sum_{j=1}^M d_{C_j}^2(P_{C_{j-1}}\cdots P_{C_1}x_{k}))$
and we use the convention $P_{C_0}P_{C_1}=I$. Note that $\mu_k=O(\lambda_k)\to 0$ as $k\to\infty$.

\end{lemma}
The proof of Lemma \ref{le:prop-spa} follows the same line of the proof of Lemma \ref{le:prop-ppa}.
For instance, part (ii) can be proved by consecutively applying property (iii) of projections in
Proposition \ref{prop:proj} (it is also proved in \cite{DN2009}).
Part (iv) can trivially be derived from (iii) by using the
Lipschitz-1 property of distance functions.

By Lemma \ref{le:prop-spa}, we find that the conclusion of Lemma \ref{le:ppa1}
holds true also for the SPA. 

\end{remark}

\subsection{Asymptotic Feasibility for Cyclic Projections}
Recall that the \emph{cyclic projection algorithm} (CPA) alternates the full sequence of gradient steps with the individual
projections on each one of the sets $C_1$, \dots, $C_N$, as follows.
\begin{equation}
\begin{cases}\tag{CPA}
x_{k,0}= x_k,\\
x_{k,j}=x_{k,j-1}-\lambda_k v_{k,j}, \quad v_{k,j} \in \partial f_{j}(x_{k,j-1}) ,\quad  j = 1,2,\cdots, N,\\
x_{k+1} = P_{C_{[k+1]}} x_{k,N},\quad  [k+1] =(k\mod M)+1.
\end{cases}
\end{equation}

Our goal is to prove the following asymptotic feasibility result that mirrors Lemma~\ref{le:ppa2}.

\begin{lemma}\label{lem:conv}
	Assume ${\rm dim}\, \HH <\infty$, (\ref{eq:Lip}), and $\lambda_k\downarrow 0$, and  that the sequence $\{x_k\}$ generated by the method of cyclic projections is bounded. Then $\{x_k\}$ is asymptotically feasible, that is, $\lim_{k\to\infty} d_C(x_k)=0$.
\end{lemma}

To prove this lemma, we need several technical claims. First, for any $x\in X $ and $q\in \{1,\dots, M\}$ define the exact $q$-cyclic projection
\begin{equation}\label{eq:defspecialP}
P_q(x): = P_{C_q} P_{C_{q-1}}\cdots P_{C_1}P_{C_M}P_{C_{M-1}}\cdots P_{C_{q+1}}(x).
\end{equation}

We next show that such cyclic projections bring the iterations closer to the feasible set in a uniform sense.

\begin{proposition}\label{prop:psi}
	Let $X$ be a nonempty compact convex subset of $\RR^n$ such that $X\setminus C \neq \emptyset$ and $X\cap C\neq \emptyset$. For each $q\in \{1,\dots, M\}$ define a function $\psi^q_X:[0,+\infty)\to [0,+\infty)$,
	\begin{equation}\label{eq:defpsi}
	\psi^q_X(\alpha):= \sup_{\substack{d(x,C)\leq \alpha\\x\in X}} d(P_q(x),C).
	\end{equation}
	The function $\psi^q_X$ is continuous and
	$\psi^q_X(\alpha)<\alpha$ for all $\alpha>0$.
\end{proposition}
\begin{proof} We assume throughout that the compact convex set $X$ and the index $q\in \{1,\dots, M\}$ are fixed and use the notation $\psi: =\psi^q_X $. We first show that $\psi(\alpha)<\alpha$ for $\alpha>0$. For any closed convex set $S$ we have by Proposition~\ref{prop:proj}(iii)
	$$
	\|x-y\|^2 \geq \|P_S(x)-y\|^2+\|P_S(x)-x\|^2,
	$$
	hence, for our setting
	\begin{align*}
	\|x-y\|^2
	& \geq \|P_{C_{q+1}}(x)-y\|^2+\|P_{C_{q+1}}(x)-x\|^2\\
	& \geq \|P_{C_{q+2}}P_{C_{q+1}}(x)-y\|^2+\|P_{C_{q+2}}P_{C_{q+1}}(x)-P_{C_{q+1}}(x)\|^2+\|P_{C_{q+1}}(x)-x\|^2\\
	& \geq \cdots \\
	& \geq \|P_{C_q}\cdots P_{C_{q+2}}P_{C_{q+1}}(x)-y\|^2+\cdots + \|P_{C_2}P_{C_1}(x)-P_{C_1}(x)\|^2+\|P_{C_1}(x)-x\|^2.
	\end{align*}
	It is evident then that if $x\notin C = \displaystyle\cap_{i=1}^M C_i$, we  have
	\begin{equation*}
	\|x-y\|^2 \geq \|P_q(x)-y\|^2+\gamma(x)\quad \forall y\in C,
	\end{equation*}
	where $\gamma(x)>0$  does not depend on $y$.
	Therefore, taking the infimum over $y\in C$, we have for every $x\notin C$
	\begin{align}\label{eq:ineq-proj}
	d^2(x,C )  & = \inf_{y\in C}\|x-y\|^2 \notag\\
	& \geq  \inf_{y\in C}\|P_q(x)-y\|^2+\gamma(x) \notag\\
	& = d^2 (P_q(x),C)+\gamma(x),
	\end{align}
	and so
	\begin{equation}\label{eq:dist-rel}
	d(x,C )>d (P_q(x),C)\; \text { for every  } \; x\notin C.
	\end{equation}
	Now let
	$$
	X_\alpha:= X\cap \{x\,|\, d(x,C)\leq \alpha\}.
	$$
	Observe that explicitly
	\begin{equation}\label{eq:psixalpha}
	\psi(\alpha) = \sup_{x\in X_\alpha} d(P_q(x), C).
	\end{equation}
	The set $X_\alpha$ is compact because it is the intersection of a compact set $X$ with a closed set $\{d(x,C)\leq \alpha\}$, and $X_\alpha$ is nonempty for every $\alpha\geq 0$ because $\emptyset \neq C\cap X= X_0\subset X_\alpha$.
	The function $d(P_q(x),C)$ is continuous in $x$, and since each of the sets $X_\alpha$ is compact and nonempty, the supremum in \eqref{eq:psixalpha} is attained, and we have
	\begin{equation}\label{eq:psimax}
	\psi(\alpha) = \max_{x\in X_\alpha} d(P_q(x), C) \quad \forall \, \alpha\geq 0.
	\end{equation}
	Hence, for every $\alpha> 0$ there exists $x_\alpha$ such that $d(x_\alpha,C)\leq \alpha$ and
	$$
	\psi(\alpha) =  d(P_q(x_\alpha), C).
	$$
	If $\psi(\alpha)=0$, then $\psi(\alpha)<\alpha$. If $\psi(\alpha)>0$, we have $x\notin C$ and from \eqref{eq:dist-rel}
	$$
	\psi(\alpha) =  d(P_q(x_\alpha), C)<d(x,C) \leq \alpha.
	$$

	We next focus on showing that $\psi$ is continuous. Since
	$$
	X_\alpha\subseteq X_\beta \quad \text{for } 0\leq  \alpha\leq \beta,
	$$
	the function $\psi(\alpha)$ is nondecreasing, and to prove its continuity it is sufficient to show
	\begin{equation}\label{eq:onesidedlimsinv}
	\liminf_{\alpha\uparrow \bar \alpha } \psi(\alpha)\geq \psi(\bar \alpha), \; \forall \bar \alpha>0, \quad \text{and} \quad \limsup_{\alpha\downarrow \bar \alpha } \psi(\alpha)\leq \psi(\bar \alpha), \; \forall \bar \alpha\geq 0.
	\end{equation}
	
	If $\psi(\bar \alpha)=0$, since $\psi$ is nondecreasing, we have $0\leq \psi(\alpha) \leq \psi(\bar \alpha) = 0$, so $\psi(\alpha)=0$ for all $\alpha\in [0,\bar \alpha]$ and the first relation in \eqref{eq:onesidedlimsinv} holds trivially. Consider the case $\psi(\bar \alpha)>0$. From \eqref{eq:psimax} we know that there exists $\bar x\in X$ such that $d(\bar x,C) \leq \bar \alpha$ and $d(P_q(\bar x),C)  = \psi(\bar \alpha)$. Let $x_0\in X_0\neq \emptyset$ (so that $d(x_0,C) =0$). Since $X$ is convex, we have $[x_0, \bar x] \subseteq X$. Let
	$$
	t_0 := \sup\{t\in [0,1]\,|\, d( x_0+ t (\bar x-x_0)), C) = 0\}.
	$$
	Since by our assumption $\psi(\bar \alpha)>0$, we have  $t_0\in [0,1)$. Now take any $t_0\leq t_1<t_2\leq 1$. We have
	\begin{align*}
	d( x_0+ t_1 (\bar x-x_0)), C)
	& \leq \| [x_0+ t_1 (\bar x-x_0))]-[x_0+\frac{t_1}{t_2} (P_C(x_0+ t_2 (\bar x-x_0)))-x_0))]\|\\
	&  = \frac{t_1}{t_2}\|x_0+ t_2 (\bar x-x_0))-(P_C(x_0+ t_2 (\bar x-x_0)))\|\\
	& =\frac{t_1}{t_2} d(x_0+ t_2 (\bar x-x_0),C)\\
	& <  d(x_0+ t_2 (\bar x-x_0),C),
	\end{align*}
	and hence $ d(x_0+ t (\bar x-x_0),C)$ is strictly increasing in $t$ for $t\in [t_0, 1]$. From this together with the continuity of the distance function we deduce that for every $t\in [0,1)$ there exists a sufficiently large $\alpha_t<\bar \alpha$ such that
	$$
	d(x_0+ t' (\bar x-x_0),C)\leq \alpha_t <\bar \alpha \quad \forall t'\in [0,t].
	$$
	
	At the same time, by the continuity of $P_q(x)$ for every $\varepsilon>0$ there exists $t\in [0,1)$ such that
	$$
	d(P_q(x_0+ t (\bar x-x_0)),C)\geq d(P_q(\bar x), C) -\varepsilon.
	$$
	This means that for every $\varepsilon>0 $ we can find $t$ and $\alpha_t$ such that
	$$
	\psi(\alpha)\geq \psi(\alpha_t)\geq d(P_q(\bar x),C)-\varepsilon \quad \forall \alpha\geq \alpha_t,
	$$
	and therefore we have the desired
	$$
	\liminf_{\alpha\uparrow \bar \alpha} \psi(\alpha)\geq \psi (\bar \alpha).
	$$
	
	It remains to show the second relation in \eqref{eq:onesidedlimsinv}. Now let $\alpha_k$ be such that $\alpha_k\downarrow \bar \alpha\geq 0$ as $k\to \infty$, and
	$$
	\lim_{k\to \infty}\psi(\alpha_k) = \limsup_{\alpha\downarrow \bar \alpha} \psi(\alpha).
	$$
	From \eqref{eq:psimax} there exists a sequence $x_k$ such that
	$$
	d(x_k,C) \leq \alpha_k,\quad \psi(\alpha_k) = d(P_q(x_k),C).
	$$
	Without loss of generality this sequence $\{x_k\}$ converges to some $\bar x\in X$. By continuity we have
	$$
	d(P_q(x_k), C) \to d(P_q(\bar x), C); \quad d(\bar x, C) = \lim_{k\to \infty}d(x_k,C)\leq \bar \alpha.
	$$
	Therefore
	$$
	\lim_{k\to \infty} \psi(\alpha_k) = d(P_q(x_k),C) = d(P_q(\bar x), C)\leq \psi(\bar \alpha).
	$$
\end{proof}

\begin{proposition}\label{prop:tiny-lambda}
Let $\{x_k\}$ be a bounded sequence obtained by means of the cyclic projections algorithm, under assumption \eqref{eq:Lip}, and $\lambda_k\downarrow 0$. Then for any $q\in \{1,\dots, M\}$ and any $\varepsilon>0$ there exists a sufficiently large $K$ such that
	$$
	\|P_q(x_k)-x_{k+M}\|\leq \varepsilon \quad \forall k\geq K, \quad (k \mod M)+1 = q,
	$$
	where $P_q$ is the exact cyclic projection operator defined by \eqref{eq:defspecialP}.
\end{proposition}
\begin{proof} Using the nonexpansivity of the projection operator (Proposition~\ref{prop:proj}~(ii)) we have
	\begin{align*}
	\|P(x_k)-x_{k+M}\|
	& = \|P_{C_q}P_{C_{q-1}}\dots P_{C_{q+1}}(x_k)-P_{C_q}(x_{k+M-1,N})\|\\
	& \leq \|P_{C_{q-1}}\dots P_{C_{q+1}}(x_k)-x_{k+M-1,N}\|\\
	& \leq \|P_{C_{q-1}}\dots P_{C_{q+1}}(x_k)-x_{k+M-1}\|+\|x_{k+M-1}-x_{k+M-1,N}\|\\
	& \leq \|P_{C_{q-1}}\dots P_{C_{q+1}}(x_k)-P_{C_{q-1}}x_{k+q-2,N}\|+\|x_{k+M-1}-x_{k+M-1,N}\|\\
	& \leq \cdots\\
	& \leq \sum_{i=1}^M \|x_{k+M-i}-x_{k+M-i,N}\|\\
	& \leq \sum_{i=1}^M \sum_{j=1}^N\|x_{k+M-i,j-1}-x_{k+M-i,j}\|\\
	& = \sum_{i=1}^M \lambda_{k+M-i} \sum_{j=1}^N\|v_{k+M-i,j}\|\\
	& \leq \sum_{i=1}^M \lambda_{k+M-i} \sum_{j=1}^N L_j\\
	& =L \sum_{i=1}^M \lambda_{k+M-i}
	\end{align*}
	where $L = \sum_{j=1}^N L_j$. Since $\lambda_k\downarrow 0$, we can always find a sufficiently large number $K$ to ensure the last term is smaller than $\varepsilon $  for all $k\geq K$.
\end{proof}

The next proposition brings us closer to the proof of Lemma~\ref{lem:conv}.

\begin{proposition}\label{prop:cconv} Assume that $\{x_k\}$ is bounded, $\lambda_k\downarrow 0$ and condition~\eqref{eq:Lip} is satisfied. Then
	$$
	\liminf_{k\to \infty}d(x_k,C)= 0.
	$$
\end{proposition}
\begin{proof}
	Assume that the claim is not true. Then for some starting point $x_0$ the sequence $\{x_k\}$ is bounded, but
	$$
	\liminf_{k\to \infty} d(x_k,C) = D>0.
	$$
	Let $\{x_{k_l}\}$ be a subsequence of $\{x_k\}$ such that
	$$
	\lim_{k\to \infty} d(x_{k_l},C) = \liminf_{k\to \infty} d(x_k,C) = D.
	$$
	Without loss of generality we may assume that $x_{k_l}\to \hat x$ and that $(k_l \mod M)+1 = q\in \{1,\dots, M\}$, so that each $x_{k_l}$ is obtained after projecting onto  $C_q$.
	
	Since the sequence $\{x_k\}$ is bounded, we can define the function $\psi = \psi^q_X$ (as in Proposition~\ref{prop:psi}) on any compact set $X$ that contains $\{x_k\}$ and some point from $C$ which we assumed to be nonempty. By the continuity of $\psi$ proved in Proposition~\ref{prop:psi} we have
	$$
	\lim_{k_l\to \infty} d(P_q(x_{k_l}),C) \leq \lim_{k\to \infty} \psi(d(x_{k_l},C)) = \psi(D) <D,
	$$
	where the last inequality follows from $D>0$ and Proposition~\ref{prop:psi}.
	
	Therefore, for sufficiently large $k_l$ we have
	$$
	d(P_q(x_{k_l}),C) \leq \psi(D)+ \frac{D-\psi(D)}{3}.
	$$
	On the other hand, using Proposition~\ref{prop:tiny-lambda} we deduce that for sufficiently large $k_l$ we also have
	$$
	\|x_{{k_l}+1}- P_q(x_{k_l})\|\leq \frac{D-\psi(D)}{3},
	$$
	hence
	\begin{align*}
	d(x_{{k_l}+1},C)
	& \leq  \|x_{k_{l}+1}-P_C(P_q(x_{k_l}))\| \\
	& \leq \|x_{{k_l}+1}- P_q(x_{k_l})\|+\|P_q(x_{k_l})-P_C(P(x_{k_l}))\|\\
	& =  \|x_{{k_l}+1}- P_q(x_{k_l})\|+d(P_q(x_{k_l}),C)\\
	& \leq \frac{D-\psi(D)}{3}+ \psi(D)+ \frac{D-\psi(D)}{3} \\
	& =  D- \frac{D-\psi(D)}{3}<D.
	\end{align*}
	Taking the lower limit, we have
	$$
	\liminf_{k_l\to \infty}d(x_{{k_l}+1},C) \leq  \psi(D)+\frac{2}{3}(D-\psi(D))<D = \liminf_{k\to \infty} d(x_{k},C),
	$$
	a contradiction.
\end{proof}

\begin{proof}[Proof of Lemma~\ref{lem:conv}] It is sufficient to show  that for any $\varepsilon>0$ there exists a sufficiently large $K$ such that for $k\geq K$ we have  $d(x_k, C) < \varepsilon$.

	Fix $\varepsilon>0$.  By Proposition~\ref{prop:psi} for every $q$ the function $\alpha-\psi_X^q(\alpha)$ is continuous and positive on the compact set $[\varepsilon/2,\varepsilon]$. Therefore, it attains its minimum, which is also positive,
	$$
	\min_{\alpha\in [\varepsilon/2, \varepsilon]}[\alpha- \psi_X^q(\alpha)] = \gamma>0.
	$$
	
	By Proposition~\ref{prop:tiny-lambda} there exists $K$ such that
	$$
	\|P^q(x_k)-x_{k+M}\|\leq \frac{\gamma}{2} \quad \forall k\geq K, \quad (k\mod M)+1 = q.
	$$
	By Proposition~\ref{prop:cconv} there exists some $k_0\geq K$ such that
	$$
	d(x_{k_0},C)< \frac{\varepsilon}{2}.
	$$
	Let $q = (k\mod M)+1$. Our goal is to show that $x_{k_0+iM}$, $i\in \NN$ never leaves the $\varepsilon$-neighbourhood of $C$. Assume the contrary. Then for some $k\geq k_0$, $(k\mod M)+1=q$ we have $d(x_k,C)\leq\varepsilon$, but $d(x_{k+M},C)>\varepsilon$. Observe that this yields
	\begin{align*}
	d(x_k,C)
	&\geq \gamma+\psi(d(x_k,C))\\
	& \geq \gamma+ d(P_q(x_k),C)\\
	& = \gamma+ \|P_q(x_k)-P_C(P_q(x_k))\|\\
	& \geq  \gamma+ \|P_C(P_q(x_k))-x_{k+M}\|-\|P_q(x_k)-x_{k+M}\|\\
	& \geq  \gamma-\frac{\gamma}{2}+d(x_{k+M},C) \\
	& >  \frac{\gamma}{2}+\varepsilon>\varepsilon,
	\end{align*}
	a contradiction.
\end{proof}

\section{From asymptotic feasibility to convergence}\label{sec:mainpf}

In the previous section we have shown that all three algorithms (cyclic, sequential and parallel projections) satisfy the asymptotic feasibility property, i.e. under the standard assumptions the sequence of iterates $(x_k)$ satisfies
$$
\lim_{k\to\infty} d_C(x_k)=0.
$$

In this final technical section we prove that this property yields the convergence of the iterative sequence to the optimal solution, which we make precise in Lemma~\ref{le:AFeasAReg}. We then briefly explain the proof of Theorem~\ref{th:main} that is based on this result and on the aforementioned property of asymptotic feasibility.

Our next statement is a useful estimate that will be utilized heavily in the subsequent analysis. Our proof is a minor modification of \cite[Lemma 2.1]{NB2001}.

\begin{lemma}\label{le:estimate}
	Let $\{x_k\}_{k=0}^\infty$ be generated by any of the three projection algorithms,
and assume that the condition \eqref{eq:Lip} is satisfied. Set $L=\sum_{i=1}^N L_i$.
	Then, for each $x\in C$, we have
	\begin{equation}\label{eq:xk+1}
	\| x_{k+1}-x\|^2\le \|x_k-x\|^2-2\lambda_k [f(x_k)-f(x)]+\lambda_k^2L^2.
	\end{equation}
\end{lemma}

\begin{proof} Let $V$ be one of the three operators considered for our projection step,
	$$
	V \in \left\{P_{C_M}\cdots P_{C_1},\quad P_{C_{[k+1]}},\quad  \sum_{i=1}^M \beta_iP_{C_i} \;
    \left(\sum_{i=1}^M \beta_i = 1, \beta_i>0\   \forall i\right)\right\}.
	$$	
	Observe that $V$ is nonexpansive, and hence for $x\in C\subseteq \Fix V$,
\begin{equation}\label{eq:xk+1M}
	\|x_{k+1}-x\|=\|Vx_{k,N}-Vx\|\le\|x_{k,N}-x\|.
	\end{equation}
On the other hand, for each $1\leq j\leq N$ and $x\in C$, we have
\begin{align*}
\|x_{k,j}-x\|^2
&=\|(x_{k,j-1}-x)-\lambda_k v_{k,j}\|^2\\
&=\|x_{k,j-1}-x\|^2-2\lambda_k\langle v_{k,j},x_{k,j-1}-x\rangle +\lambda_k^2\|v_{k,j}\|^2.
\end{align*}
Using the Lipschitz bound (\ref{eq:Lip}) and the subdifferential inequality
$$f_j(x)\geq f_j(x_{k,j-1})+\langle v_{k,j},x-x_{k,j-1}\rangle $$
we obtain
$$\|x_{k,j}-x\|^2\leq \|x_{k,j-1}-x\|^2-2\lambda_k[f_j(x_{k,j-1})-f_j(x)]+\lambda_k^2L_j^2. $$
Adding up the above inequalities over $j=1,2,\cdots,N $ yields
\begin{align}\label{eq:xkN}
\|x_{k,N}-x\|^2
&\le \| x_k-x\|^2-2\lambda_k\sum_{j=1}^N [f_j(x_{k,j-1})-f_j(x)]+\lambda_k^2\sum_{j=1}^N L_j^2 \nonumber\\
&=\|x_k-x\|^2-2\lambda_k[f(x_k)-f(x)] \nonumber\\
&\quad -2\lambda_k \sum_{j=1}^N[f_j(x_{k,j-1})-f_j(x_k)]+\lambda_k^2 \sum_{j=1}^N L_j^2.
\end{align}

	
	In view of \eqref{eq:xk+1M}, to show \eqref{eq:xk+1} it remains to bound the last two terms in (\ref{eq:xkN}).
	From the Lipschitz bound \eqref{eq:Lip} we have
	$$
	f_j(x_{k,j-1})-f_j(x_k) \geq  -L_j\|x_{k,j-1}-x_k\|.$$
Also observe that
	\begin{align*}
	\|x_{k,j-1}-x_k\|
	& =\left\|\sum_{l=1}^{j-1} (x_{k,l}-x_{k,l-1})\right\|
	=\left\|\sum_{l=1}^{j-1} \lambda_k v_{k,l}\right\|
	\le\lambda_k \sum_{l=1}^{j-1} L_l,
	\end{align*}
	where $v_{k,l} \in \partial f_l(x_{k,l-1})$. We hence obtain the desired bound
	\begin{align}\label{eq:desiredbound}
	-2 \lambda_k\sum_{j=1}^N (f_j(x_{k,j-1})- f_j(x_{k}))&+\lambda_k^2 \sum_{j=1}^N L_j^2 \nonumber\\
& \leq 2 \lambda^2_k \sum_{j=1}^N L_j  \left(\sum_{l=1}^{j-1} L_l\right) +\lambda^2_k\sum_{j=1}^NL_j^2 = \lambda^2_k L^2.
	\end{align}
	Now combining \eqref{eq:xk+1M} with \eqref{eq:xkN} and \eqref{eq:desiredbound} we obtain \eqref{eq:xk+1}.
\end{proof}

\begin{lemma}\label{le:AFeasAReg}	Let $\{x_k\}$ be a sequence generated by one of the three projection algorithms, and assume that $\{x_k\}$ is bounded and asymptotically feasible, i.e.,
	\begin{equation}\label{eq:ar1}
	\lim_{k\to\infty} d_C(x_k)=0.
	\end{equation}
	Then the following conclusions are satisfied:
	\begin{enumerate}
		\item[(i)]
		$\{x_k\}$ is asymptotically regular, that is, $\lim_{k\to\infty}\|x_{k+1}-x_k\|=0$;
		\item[(ii)]
		$\liminf_{k\to\infty} f(x_k)=f^*$, which implies that $\liminf_{k\to\infty} d_{S^*}(x_k)=0$.
	\end{enumerate}
\end{lemma}
\begin{proof} From Lemma~\ref{le:estimate} we have
	\begin{equation}\label{BasicIneq}
	\|x_{k+1}-x\|^2\le\|x_k-x\|^2-2\lambda_k[f(x_k)-f(x)]+\lambda_k^2L^2,\quad x\in C.
	\end{equation}
	(i)
	Take a subsequence $\{x_{k_i}\}$ of $\{x_k\}$ such that
	\begin{equation}\label{lim}
	\limsup_{k\to\infty}\|x_{k+1}-x_k\|=\lim_{i\to\infty}\|x_{k_i+1}-x_{k_i}\|.
	\end{equation}
	With no loss of generality, we may assume $x_{k_i}\to\hat{x}$; then $\hat{x}\in C$ by (\ref{eq:ar1}).
	Use (\ref{BasicIneq}) with $k$ and $x$ replaced with $k_i$ and $\hat{x}$, respectively, to get
	(noting that $\lambda_{k_i}\to 0$ and $(f(x_{k_i}))$ is bounded)
	\begin{equation*}
	\|x_{k_i+1}-\hat{x}\|^2\le\|x_{k_i}-\hat{x}\|^2-2\lambda_{k_i}[f(x_{k_i})-f(\hat{x})]+\lambda_{k_i}^2L^2\to 0.
	\end{equation*}
	It turns out that $x_{k_i+1}\to \hat{x}$. Returning to (\ref{lim}), we immediately find that $\|x_{k+1}-x_k\|\to 0$.

	(ii)
	We have a subsequence $\{x_{k_i}\}$ of $\{x_k\}$ such that
	$$\liminf_{k\to\infty} f(x_k)=\lim_{i\to\infty} f(x_{k_i}).$$
	Due to boundedness, we may also assume $x_{k_i}\to x'$. By part (i), $x'\in C$
	and we therefore $\liminf_{k\to\infty} f(x_k)=f(x')\ge f^*$.
	
	On the other hand, if
	$\liminf_{k \to \infty}f(x_k)>f^*$, then there exist some $\varepsilon_0>0$ and $k'\ge 0$ such that
	$f(x_k)>f^*+\varepsilon_0$ and $\lambda_kL^2<\varepsilon_0$ for all $k\ge k'$.
	It then turns out from (\ref{eq:xk+1}) that, for $x\in S^*$ and $k\geq k'$,
	\begin{align*}
	\varepsilon_0\lambda_k\le\|x_k-x\|^2-\|x_{k+1}-x\|^2.
	\end{align*}
	This implies that the series $\sum_{k=k'}^\infty \lambda_k<\infty$, which contradicts (\ref{eq:lambda1}).
	So we must have $\liminf_{k \to \infty}f(x_k)\le f^*$.
\end{proof}

We finish this section with the proof of Theorem~\ref{th:main}. The proof that we provide below contains a point
that is essentially different from that of \cite{DN2009}, which makes us successfully remove the
assumption in  \cite[Theorem 1]{DN2009} and  in \cite[Proposition 2.3]{NB2001} that the optimal solution set $S^*$ be bounded.
Note that this condition is equivalent to (\cite{R1997}) the condition that the objective function $f$ satisfies the coercivity
property: $f(x)\to\infty$ as $\|x\|\to\infty$.

\begin{proof}[Proof of Theorem~\ref{th:main}] It is sufficient to prove that the following two claims are true under the conditions of Theorem~\ref{th:main} (that $\HH $ is finite-dimensional, the sequence $\{x_k\}$ is bounded, and the two conditions \eqref{eq:lambda1} and \eqref{eq:Lip} are satisfied):
	\begin{enumerate}
		\item[(i)]
		$\lim_{k\to\infty} d_{S^*}(x_k)=0$; in other words, every cluster point of $\{x_k\}$ is an optimal solution
		of (\ref{min:f});		
		\item[(ii)] $\lim_{k\to\infty} f(x_k)=f^*$.
	\end{enumerate}
	Observe that (ii) is an immediate consequence of (i) due to the continuity of the objective function $f$. We hence focus on proving (i).
	
Observe that for each $\varepsilon>0$ and each $k\in \NN$ exactly one of the two possibilities holds:
\begin{enumerate}
	\item[(1)] $f(x_k)>f^*+\varepsilon$ and
	\item[(2)] $f(x_k)\le f^*+\varepsilon$.
\end{enumerate}	
	First consider case (1). By (\ref{eq:xk+1}), we get
	\begin{equation*}
	\|x_{k+1}-x\|^2\le \|x_k-x\|^2-2\lambda_k [f(x_k)-f^*]+\lambda_k^2L^2.
	\end{equation*}
	It turns out that
	\begin{align*}
	d^2_{S^*}(x_{k+1})&\le d^2_{S^*}(x_k)-2\lambda_k [f(x_k)-f^*]+\lambda_k^2L^2\\
	&<d^2_{S^*}(x_k)-\lambda_k (2\varepsilon-\lambda_kL^2).
	\end{align*}
	Since $\lambda_k\to 0$, we may assume $\lambda_kL^2<\varepsilon$. We then get for sufficiently large $k$
	\begin{equation}\label{d2}
	d^2_{S^*}(x_{k+1})<d^2_{S^*}(x_k)-\varepsilon\lambda_k.
	\end{equation}
	In particular,
	\begin{equation}\label{d2S}
	d_{S^*}(x_{k+1})<d_{S^*}(x_k).
	\end{equation}
	
	We now turn to consider case (2) which is valid infinitely often as $\liminf_{k\to\infty}f(x_k)=f^*$.
	Define
	\begin{equation}\label{var}
	\varphi_k(\varepsilon):=\sup\{d_{S^*}(x_j): j\ge k, \  f(x_j)\le f^*+\varepsilon\}.
	\end{equation}
	It is easy to see that $\varphi_k(\varepsilon)$ is decreasing in $k$ and $\varepsilon>0$, respectively.
	Let
	\begin{equation}\label{varp}
	\varphi(\varepsilon):=\lim_{k\to\infty}\varphi_k(\varepsilon).
	\end{equation}
	It is not hard to find that
	\begin{equation}\label{limv}
	\lim_{\varepsilon\downarrow 0}\varphi(\varepsilon)=0.
	\end{equation}
	
	Indeed, if $\eta:=\lim_{\varepsilon\downarrow 0}\varphi(\varepsilon)>0$, we can find
	$\varepsilon_0>0$ such that $\varphi(\varepsilon)>\frac12\eta$ for all $0<\varepsilon<\varepsilon_0$.
	Upon taking a positive sequence $\varepsilon_0>\varepsilon_i\to 0$, we get a subsequence $\{x_{k_i}\}$ of $\{x_k\}$
	such that
	$$f(x_{k_i})\le f^*+\varepsilon_i\quad {\rm and}\quad d_{S^*}(x_{k_i})\ge \frac12\eta$$
	for all $i$.
	Assuming that $\{x_{k_i}\}$ converges to some $\bar x\in C$, we obtain the following contradiction:
	$$f(\bar x)\le f^* \  \mbox{(thus, $\bar x\in S^*$)}\quad {\rm and}\quad d_{S^*}(\bar x)\ge \frac12\eta>0
	\ \mbox{(thus, $\bar x\not\in S^*$)}.$$
	Hence, (\ref{limv}) is proven.
	
	To prove $d_{S^*}(x_k)\to 0$, noting Lemma~\ref{le:ppa2} (for parallel projections) Lemmas~\ref{lem:conv} (for cyclic projections) and \cite[Proposition~1]{DN2009} (for sequential projections)   together with Lemma \ref{le:AFeasAReg} and (\ref{varp}), we can take $k_0$ such that
	\begin{enumerate}
		\item[(i)]
		$d_{S^*}(x_{k_0})<\varepsilon$;
		\item[(ii)]
		$\lambda_kL^2<\varepsilon$ and $d_C(x_k)<\frac14\varepsilon$ for all $k\ge k_0$;
		\item[(iii)]
		$\|x_{k+1}-x_k\|<\frac12\varepsilon$ for all $k\ge k_0$;
		\item[(iv)]
		$\varphi_k(\varepsilon)<\varphi(\varepsilon)+\frac12\varepsilon$ for all $k\ge k_0$.
	\end{enumerate}
	We next prove by induction that
	\begin{equation}\label{dS*xk}
	d_{S^*}(x_{k_0+i})<\varphi(\varepsilon)+\varepsilon
	\end{equation}
	for each $i\ge 0$. This holds trivially when $i=0$.
	Upon assuming (\ref{dS*xk}) for $i$, we shall prove it for $i+1$. As a matter of fact,
	if $f(x_{k_0+i})\ge f^*+\varepsilon$, then by (\ref{d2S}), we get
	$d_{S^*}(x_{k_0+i+1})<d_{S^*}(x_{k_0+i})<\varphi(\varepsilon)+\varepsilon$
	and (\ref{dS*xk}) holds for $i+1$.
	If $f(x_{k_0+i})\le f^*+\varepsilon$, then using (iii) and (iv), and the definition (\ref{var}) of $\varphi_{k_0+i}$, we obtain
	\begin{align*}
	d_{S^*}(x_{k_0+i+1})&\le d_{S^*}(x_{k_0+i})+\|x_{k_0+i+1}-x_{k_0+i}\|\\
	&\le \varphi_{k_0+i}(\varepsilon)+\frac12\varepsilon\\
	&<\varphi(\varepsilon)+\varepsilon.
	\end{align*}
	and (\ref{dS*xk}) holds as well.
	
	Finally, (\ref{dS*xk}) implies that $\limsup_{k\to\infty}d_{S^*}(x_{k})\le\varphi(\varepsilon)+\varepsilon$
	which in turn implies that $\lim_{k\to\infty}d_{S^*}(x_{k})=0$ since $\varepsilon>0$ is arbitrary.
	
\end{proof}

\section{Generalizations}

In this section we discuss the extent to which our results can be directly generalized to the infinite-dimensional Hilbert space setting, and provide several extensions of the proposed algorithms.

\subsection{Infinite-dimensional real Hilbert space}\label{ssec:inf}

We first consider the infinite-dimensional setting. We clarify the generalizations of our main technical results in the next remark and then present the generalization explicitly in Theorem~\ref{th:ppa2}.

\begin{remark}\label{re:weak-feasibility}	Note that Lemmas~\ref{le:ppa2} and \ref{le:ppa1} remain valid in the infinite-dimensional case. In the proof of Lemma~\ref{le:ppa1},
	we may assume that subsequence $x_{k_l}\to\hat{x}$ weakly. Using the weak lower-semicontinuity of the convex function
	$\sum_{j=1}^M\beta_j d_{C_j}^2$, we still get (\ref{s}).
	
	In Lemma \ref{le:ppa2}, if ${\rm dim}\, \HH =\infty$, it turns out that $x^*\in C$
	for all $x^*\in\omega_w(x_k)$, the set of all weak cluster points of $\{x_k\}$.
	Indeed, if $x_{k_i}\to x^*$ weakly, then the weak lower-semicontinuity of the
	distance function $d_C$ implies that
	$$d_C(x^*)\le\liminf_{i\to\infty}d_C(x_{k_i})=\lim_{k\to\infty}d_C(x_k)=0.$$
	Hence, $x^*\in C$.
	
	Lemma \ref{le:AFeasAReg}(ii) also remains valid for the case of parallel projections. In fact, in this case, we have $x_{k_i}\to x'$ weakly, and
	from the proof of Lemma \ref{le:AFeasAReg}(ii),
	we get $$\liminf_{k\to\infty}f(x_{k})=\lim_{i\to\infty} f(x_{k_i})\ge f(x')\ge f^*$$
	as $x'\in C$.
	
	It is unclear if the asymptotic regularity of $\{x_k\}$ (i.e., Lemma \ref{le:AFeasAReg}(i))
	remains valid if ${\rm dim}\, \HH =\infty$.

\end{remark}

Based on Remark~\ref{re:weak-feasibility} we can state the following (incomplete) result in a general Hilbert space which may be infinite-dimensional.

\begin{theorem}\label{th:ppa2}
	Let $\{x_k\}$ be the sequence generated by the parallel projection algorithm 	in a general Hilbert space $\HH $.
	Assume (\ref{eq:lambda1}) and (\ref{eq:Lip}). Then
	there exists a subsequence $\{x_{k_j}\}$ of $\{x_k\}$ such that
	$\{x_{k_j}\}$ converges weakly to an optimal solution $x^*\in S^*$, and
	$\{f(x_{k_j})\}$ converges to the optimal value $f^*$.
	If, in addition, the limit of the full sequence $\{f(x_k)\}$ exists as $k\to\infty$, then the full sequence
	$\{x_k\}$ converges weakly to the optimal solution $x^*$, and
	$\{f(x_{k})\}$ converges to the optimal value $f^*$.
	
\end{theorem}
\begin{proof}
	By Remark \ref{re:weak-feasibility}, we have a subsequence $\{x_{k_j}\}$ of $\{x_k\}$ such that
	\begin{equation}\label{eq:xkj}
	\lim_{j\to \infty}f(x_{k_j})=\liminf_{k\to\infty}f(x_k)=f^*.
	\end{equation}
	We may also assume that $x_{k_j}\to x^* $ weakly as $j\to \infty$. Notice that $x^*\in C$ again by Remark \ref{re:weak-feasibility}.
	So the  weak lower-semicontinuity, we get
	$$f^*\le f(x^*)\le \liminf_{k\to\infty}f(x_k)=f^*.$$
	It turns out that $f(x^*)=f^*$.
	
\end{proof}

\subsection{Relaxing the Assumptions}

We have mentioned earlier that it is possible to replace the Lipschitz condition~\ref{eq:Lip} by the assumption that the components of the objective functions are bounded on bounded sets.

\begin{remark}\label{re:1}
	Theorem \ref{th:main} removes the boundedness assumption of the solution set $S^*$ of (\ref{min:f})
	of \cite[Theorem 1]{DN2009}.
	It is an open question whether or not the full sequence $\{x_k\}$ converges
	under the conditions in Theorem \ref{th:main}, even if  we further assume
	(a) $S^*$ is bounded and (b) $\{\lambda_k\}$ satisfies the following stronger condition:
	\begin{equation}\label{sum}
	\sum_{k=0}^\infty\lambda_k=\infty,\quad \sum_{k=0}^\infty\lambda^2_k<\infty.
	\end{equation}
\end{remark}
All information that is available is given by the inequality
\begin{equation}\label{xk+1}
\|x_{k+1}-x^*\|^2\le \|x_k-x^*\|^2-2\lambda_k [f(x_k)-f^*]+\lambda_k^2L^2,
\end{equation}
where $x^*\in S^*$.  Setting $\alpha_k=\|x_k-x^*\|^2$, $\beta_k=\lambda_kL$, and $\mu_k=2(f(x_k)-f^*)/L$,
we can rewrite (\ref{xk+1}) as
\begin{equation}\label{alp}
\alpha_{k+1}\le \alpha_k-2\beta_k\mu_k+\beta_k^2,
\end{equation}
where $\{\mu_k\}$ satisfies the condition:
\begin{equation}\label{mu}
\mu_k\in \mathbb{R}\ (\forall k\ge 0)\  {\rm and}\ \lim_{k\to\infty}\mu_k=0
\end{equation}
and $\{\beta_k\}$ satisfies the conditions:
\begin{equation}\label{bet}
\beta_k\ge 0\ (\forall k\ge 0),\   \{\beta_k\}\not\in\ell_1,\  \{\beta_k\}\in\ell_2.
\end{equation}
However, the conditions (\ref{mu}) and (\ref{bet}) are insufficient to imply
from (\ref{alp}) that
$\lim_{k\to\infty}\alpha_k$ exists, as shown by the example below.

\begin{example}
	Take $\alpha_k=|\sin\log k|$ ($k\ge 1$), $\beta_k=\frac1{k^{\alpha}}$, with $\alpha\in (0,1)$ and
	$2\alpha>1$ (e.g. $\alpha=\frac23$).
	Let $\mu_k$ satisfy the equation:
	$$\frac1{k}=-\beta_k\mu_k+\beta^2_k.$$
	In other words,
	$$\mu_k=-\frac{1}{k^{1-\alpha}}+\frac{1}{k^\alpha}\to 0.$$
	(Note that $\mu_k<0$ for all $k$.)
	Then $\{\mu_k\}$ and $\{\beta_k\}$ satisfy (\ref{mu}) and (\ref{bet}), respectively.
	Also, $\{\alpha_k\}$ satisfies (\ref{alp}). As a matter of fact, we have
	\begin{align*}
	\alpha_{k+1}-\alpha_k
	&=|\sin\log (k+1)|-|\sin\log k|\\
	&\le |\sin\log (k+1)-\sin\log k|\\
	&\le |\log (k+1)-\log k|=\log (1+\frac1{k})\\
	&\le\frac1{k}=-\beta_k\mu_k+\beta^2_k.
	\end{align*}
	However, $\{\alpha_k\}$ is divergent (this is easy to see from observing that $\log x - \log (x+1) = \log \frac{x}{x+1}$ converges to zero, and that $\log x\to \infty$ as $x\to \infty$; hence the expression $|\sin \log k|$ takes values infinitely close to $0$ and $1$ as $k$ goes to infinity).
	
\end{example}
\begin{remark}
	A sufficient condition for  $\{\alpha_k\}$ to be convergent is that $\mu_k\ge 0$ for all sufficiently large $k$.
	In this case, the inequality (\ref{alp}) implies
	\begin{equation}\label{alph}
	\alpha_{k+1}\le \alpha_k+\beta_k^2
	\end{equation}
	for all large enough $k$. This together with the assumption of $\{\beta_k\}\in\ell_2$
	is sufficient to imply that $\lim_{k\to\infty}\alpha_k$ exists.
	
	Returning to the sequence $\{x_k\}$, we can't get any convergence information from the inequality
	(\ref{xk+1}) since we do not know for what $k$, $x_k$ is feasible (i.e., $x_k\in C$); in other words,
	we do not know for what $k$, $f(x_k)-f^*\ge 0$.
\end{remark}

The following is another partial answer to the open question set forth in Remark \ref{re:1}.

\begin{proposition}\label{prop:finite-cluster}
	Under the conditions of Theorem \ref{th:main}, if $\{x_k\}$ has at most finitely many cluster points, then
	$\{x_k\}$ converges to an optimal solution of (\ref{min:f}). In particular, if $f$ is strictly convex,
	then $\{x_k\}$ converges to the unique optimal solution of (\ref{min:f}).
\end{proposition}

\begin{proof}
Assume that $\{x_k\}$ has $m$ cluster points, where $m\ge 1$ is an integer.
We shall prove $m=1$ by contradiction. Suppose $m>1$ and let $\xi_1,\cdots,\xi_m$
be the $m$ distinct cluster points of $\{x_k\}$. Let $\varepsilon$ satisfy the condition:
$$0<\varepsilon<\frac{\min\{\|\xi_i-\xi_j\|: 1\le i\not=j\le m\}}{\max\{m+1,3\}}.$$
	Define
	$$N_i:=\{k\in\mathbb{N}: \|x_k-\xi_i\|<\varepsilon\},\quad i=1,2,\cdots, m.$$
It is easy to see that $\{N_i\}$ are mutually disjoint: $N_i\cap N_j=\emptyset$ for all $i\not=j$, and
	$$\mathbb{N}\setminus \cup_{i=1}^m N_i$$
	is at most a finite set. Therefore, we may assume that
	$$\mathbb{N}=\cup_{i=1}^m N_i.$$
We then take an integer $k_0$ big enough so that
	\begin{equation}\label{xk+1-}
	\|x_{k+1}-x_k\|<\varepsilon,\quad k\ge k_0.
	\end{equation}
Next we take a smallest integer $k'>k_0$  such that
		\begin{equation}\label{xk'-x}
	\|x_{k'}-\xi_1\|<\varepsilon.
	\end{equation}
	Now since $k'-1\in N_{i'}$ for some $i'>1$ (i.e., $\|x_{k'-1}-\xi_{i'}\|<\varepsilon$), we arrive at the contradiction:
	\begin{align*}
	3\varepsilon<\|\xi_1-\xi_{i'}\|\le\|\xi_1-x_{k'}\|+\|x_{k'}-x_{k'-1}\|+\|x_{k'-1}-\xi_{i'}\|<3\varepsilon.
	\end{align*}
Consequently, we must have $m=1$; equivalently, the full sequence $\{x_k\}$ converges.
\end{proof}

\begin{remark}
The conclusions of Proposition \ref{prop:finite-cluster} hold true in a more general
case where the sequence $\{x_k\}$ has a set of cluster points
which is strongly isolated in the sense that
\begin{equation*}
\delta:=\inf\{\|\xi-\eta\|: \xi,\eta\in\omega(x_k),\   \xi\not=\eta\}>0.
\end{equation*}
Here $\omega(x_k)$ is the set of cluster points of $\{x_k\}$.

Indeed, let $0<\varepsilon<\frac12\delta$ and let $k_0$ satisfy (\ref{xk+1-}).
Due to the compactness of $\overline{\{x_k\}_{k=1}^\infty}$, we can find an integer $m\ge 1$
with the property
$$\bigcup_{i=1}^m B(x_{i},\varepsilon)\supset \overline{\{x_k\}_{k=1}^\infty}\supset \omega(x_k).$$
We may assume  $\#\omega(x_k)>m$ (the case where $\#\omega(x_k)\le m$
being proven in Proposition \ref{prop:finite-cluster}).
Consequently, there exists a ball $B(x_{i},\varepsilon)$ (for some $1\le i\le m$) which
contains at least two points of $\omega(x_k)$, $\xi_1$ and $\xi_2$ (say).
It turns out from the definition of $\delta$ that
$$\delta\le \|\xi_1-\xi_2\|<2\varepsilon.$$
This is a contradiction as $2\varepsilon<\delta$.

\end{remark}

\begin{proposition}
	Under the conditions of Theorem \ref{th:main}, if we assume $M=1$ (i.e., $C=C_1$)
	and $(\lambda_k)\in\ell_2$ (This is considered in \cite{NB2001}), then
	$\{x_k\}$ converges to an optimal solution of (\ref{min:f}).
\end{proposition}

\begin{proof}
	In this case, every $x_{k+1}=P_Cx_{k,N}$ is feasible (i.e., $x_{k+1}\in C$). Hence $f(x_k)-f^*\ge 0$
	and the inequality (\ref{xk+1}) implies that
	\begin{equation*}
	\|x_{k+1}-x^*\|^2\le \|x_k-x^*\|^2+\lambda_k^2L^2.
	\end{equation*}
	Therefore, the convergence of the series $\sum_{k=0}^\infty\lambda_k^2<\infty$
	implies that $\lim_{k\to\infty}\|x_k-x^*\|$ exists for each $x^*\in S^*$,
	which in turns implies that $\{x_k\}$ converges since we have proved that
	every cluster point of $\{x_k\}$ is in $S^*$.
\end{proof}

Consider the case of the parallel projections algorithm where the stepsizes are not diminishing.
We have the result below.

\begin{proposition}
	Let $\{x_k\}$ be generated by the parallel projection algorithm with nondiminishing stepsize sequence $\{\lambda_k\}$. Then
	\begin{equation}\label{eq:f}
	\liminf_{k\to\infty} f(x_k)\le f^*+\frac12\overline{\lambda} L^2,
	\end{equation}
	where $\overline{\lambda}=\limsup_{k\to\infty}\lambda_k.$
\end{proposition}
\begin{proof}
	Suppose (\ref{eq:f}) were not true; then
	\begin{equation}\label{eq:f(xk)}
	\liminf_{k\to\infty} f(x_k)> f^*+\frac12\overline{\lambda} L^2.
	\end{equation}
	For any $\varepsilon>0$, find $\hat{x}\in C$ and $k\ge 1$ such that
	\begin{itemize}
		\item $f^*>f(\hat{x})-\varepsilon$;
		\item $f(x_k)>\underline{f}-\varepsilon$ ($\underline{f}:=\liminf_{k\to\infty} f(x_k)$) for all $k\ge k_0$;
		\item $\lambda_k<\overline{\lambda}+\frac{2\varepsilon}{L^2}$ for all $k\ge k_0$.
	\end{itemize}
	These combining with (\ref{eq:f(xk)}) imply that
	\begin{equation}\label{eq:f(xk2)}
	f(x_k)> f(\hat{x})+\frac12\lambda_kL^2+\varepsilon,\quad k\ge k_0.
	\end{equation}
	Now applying (\ref{eq:xk+1}) and using (\ref{eq:f(xk2)}) we further obtain, for all $k\ge k_0$,
	\begin{align*}
	\|x_{k+1}-\hat{x}\|^2&\le \|x_{k}-\hat{x}\|^2-2\lambda_k[f(x_k)-f(\hat{x})]+\lambda_k^2L^2\\
	&\le \|x_{k}-\hat{x}\|^2-2\varepsilon\lambda_k.
	\end{align*}
	It turns out that
	\begin{equation*}
	\sum_{i=k_0}^k\lambda_i\le\frac{\|x_{k_0}-\hat{x}\|^2}{2\varepsilon},\quad k\ge k_0.
	\end{equation*}
	Hence, $\{\lambda_k\}$ must be in $\ell_1$, a contradiction to the assumption that
	$\{\lambda_k\}$ is nondiminishing.
\end{proof}

\begin{remark}[Unrestricted and random projections] Note that the cyclic projection algorithm can be generalized to an \emph{unrestricted} version, where the order of the projections is not sequential, but is determined by a mapping $\phi: \NN\to \{1,\dots, M\}$, defined so that each  of the sets $C_1,C_2,\dots, C_M$ feature in this algorithm infinitely many times. If there is a uniform bound on the gap between the number of steps separating the next nearest appearance of the same set in the sequence, then our analysis of the method of cyclic projections can be generalized to include this version of the method. It remains to be seen if the convergence still holds without this assumption, and whether some probabilistic bounds can be obtained for a randomized version of the method.
\end{remark}

\section{Relaxed Projection Algorithms}

Here we briefly outline ideas of relaxed projection approaches that can be used whenever the projections may be expensive or unavailable, but an approximation is reasonably easy to compute. This can be considered in a general framework of cutters (e.g. see \cite{Cegielski}). We consider the most popular implementation of cutters via the subgradients of constraint functions.

Assume each $C_i$ is a level set of a convex function, that is,
\begin{equation}\label{eq:4:setCQ}
C_i=\{x\in \HH : c_i(x)\le 0\},\quad 1\le i\le M,
\end{equation}
where $c_i: \HH \to \mathbb{R}\cup\{\infty\}$ is a convex function which
is subdifferentiable on an open convex set that contains $C_i$.
Recall that the subdifferential of $c_i$ at $x\in {\rm dom}\, c_i$ is defined by
$$\partial c_i(x)=\{z\in \HH : c_i(w)\ge c_i(x)+\langle w-x,z\rangle,\quad w\in \HH \}.$$
In this setting we are able to replace projections onto the $C_i's$ with 
projections onto half-spaces, which then have closed formulae. 

We consider the relaxed parallel projection algorithm (RPPA) and the relaxed sequential projection algorithm (RSPA)
which generate a sequence  $\{x_k\}$ by the following iteration processes:

\begin{equation}
\begin{cases}\tag{RPPA}
	x_{k,0}=x_k, \\[0.1cm]
	x_{k,j}=x_{k,j-1}-\lambda_k v_{k,j},\quad  v_{k,j}\in\partial f_{j}(x_{k,j-1}),\   j = 1,2,\cdots, N,  \\[0.1cm]
	x_{k+1}=\sum_{i=1}^M \beta_i P_{C^k_i}x_{k,N},
\end{cases}
\end{equation}
and, respectively,
\begin{equation}
\begin{cases}\tag{RSPA}
	x_{k,0}=x_k, \\[0.1cm]
	x_{k,j}=x_{k,j-1}-\lambda_k v_{k,j},\quad  v_{k,j}\in\partial f_{j}(x_{k,j-1}),\   j = 1,2,\cdots, N,  \\[0.1cm]
	x_{k+1}=P_{C^k_M}\cdots P_{C^k_1}x_{k,N},
\end{cases}
\end{equation}
where $\beta_i>0$ and sum to one: $\sum_{i=1}^M\beta_i=1$, and $C^k_{i}$ is a half-space defined by
\begin{equation}\label{Cki}
C^k_{i}:=\{x\in \HH : c_{i}(x_k)+\langle \xi_i^k,x-x_k\rangle\le 0\},\quad \xi_i^k\in\partial c_{i}(x_k)
\end{equation}
for $i=1,\cdots, M$. Note that $C^k_{i}\supset C_{i}$ for each $i$ and $k$. Indeed, if $x\in C_{i}$ (i.e., $c_{i}(x)\le 0$), then by the
subdifferential inequality, we get
$$0\ge c_{i}(x)\ge  c_{i}(x_k)+\langle \xi_i^k,x-x_k\rangle.$$
This shows that $x\in C^k_{i}$.

Define
$$T_k:=\sum_{i=1}^M \beta_i P_{C^k_i}\ \text{(for RPPA)}\quad {\rm or}\quad T_k:=P_{C^k_M}\cdots P_{C^k_1}\   \text{(for RSPA)}.$$
Then $T_k$ is nonexpansive (as a convex combination (or composite) of projections). Moreover, we can rewrite $x_{k+1}=T_k x_{k,N}$.
Note that
$$\Fix T_k=\bigcap_{i=1}^M C^k_i\supset \bigcap_{i=1}^M C_i=C.$$

First we consider the sequence $\{x_k\}$ generated by the (RPPA). An immediate analysis shows that 
Lemma \ref{le:estimate} remains valid for the (RPPA), and
Lemma \ref{le:prop-ppa} valid for the (RPPA) as well with $C_j$ replaced with
$C_j^k$ for each $j$. We now verify Lemma \ref{le:ppa1} for the (RPPA).
As a matter of fact, we can follow the same way of the proof of Lemma \ref{le:prop-ppa},
and (\ref{sum-j}) to get
\begin{equation}\label{eq:rppa1}
\sum_{j=1}^N\beta_j d_{C^{k_l}_j}^2(x_{k_l})\le d_C^2(x_{k_l})-d_C^2(x_{k_l+1})+O(\lambda_{k_l})<\frac1{l}+O(\lambda_{k_l})\to 0.
\end{equation}
Since $\{x_k\}$ is a bounded sequence in a finite dimensional space, we may assume that $x_{k_l}\to\hat{x}$.
We then get (for distance functions are 1-Lipschitz continuous)
\begin{equation}\label{eq:rppa2}
d_{C^{k_l}_j}(\hat{x})\to 0,\quad j=1,2,\cdots, M.
\end{equation}
It follows that there exists some $z_{j,l}\in C^{k_l}_j$ such that
\begin{equation}\label{eq:rppa3}
\|z_{j,l}-\hat{x}\|\to 0\quad (l\to\infty), \   j=1,2,\cdots, M.
\end{equation}
Since $z_{j,l}\in C^{k_l}_j$, we have
\begin{equation}\label{eq:rppa4}
c_{j}(x_{k_l})+\langle \xi_j^{k_l},z_{j,l}-x_{k_l}\rangle\le 0.
\end{equation}
Noting the boundedness of $(\xi_j^{k})$ and using the facts $z_{j,l}\to \hat{x}$ and $x_{k_l}\to\hat{x}$,
we immediately obtain that the second term in the last relation tends to zero as $l\to\infty$.
Consequently, we get $c_j(\hat{x})\le 0$ for each $j$; that is, $x\in C$.

Next consider the sequence  $\{x_k\}$ generated by the (RSPA).
In this case we still have Lemma \ref{le:estimate} valid for the (RSPA). Moreover,
Lemma \ref{le:prop-spa} remains valid for the (RPPA)  with $C_j$ replaced with
$C_j^k$ for each $j$. To see  Lemma \ref{le:ppa1} is also valid for the (RSPA),
we find that the relation \eqref{eq:rppa1} for the (RPPA) is replaced by 
the relation below for the (RSPA):
\begin{equation}\label{eq:rspa1}
\sum_{j=1}^M d^2_{C_j^{k_l}}(P_{C^{k_l}_{j-1}}\cdots P_{C^{k_l}_1}x_{k_l})
\le d_C^2(x_{k_l})-d_C^2(x_{k_l+1})+O(\lambda_{k_l})<\frac1{l}+O(\lambda_{k_l})\to 0.
\end{equation}
Assuming $x_{k_l}\to \hat{x}$ as $l\to\infty$, we get
\begin{equation*}
d_{C_j^{k_l}}(P_{C^{k_l}_{j-1}}\cdots P_{C^{k_l}_1}\hat{x})\to 0\quad {\rm as}\ l\to\infty
\end{equation*}
for each $j=1,2,\cdots, M$. It then turns out that we can find $z_{j,k_l}\in C_j^{k_l}$
such that
$z_{j,k_l}\to \hat{x}$ as $l\to\infty$ for each $j=1,2,\cdots, M$. Namely, 
\eqref{eq:rppa2}-\eqref{eq:rppa3} remain valid. Then again
from \eqref{eq:rppa4}, we derive that $\hat{x}\in C$.

Finally, the proof of Theorem \ref{th:main} can easily be repeated to prove the convergence of
the (RPPA) and (RSPA), which is stated below.

\begin{theorem}\label{th:rppa}
	Let $\{x_k\}$ be a sequence generated either by the (RPPA) or by the (RSPA).
	Assume ${\rm dim}\, \HH <\infty$ and $\{x_k\}$ is bounded. Then we have
	Assume also (\ref{eq:lambda1}) and (\ref{eq:Lip}).
	\begin{enumerate}
		\item[(i)]
		$\lim_{k\to\infty} d_{S^*}(x_k)=0$; in other words, every cluster point of $\{x_k\}$ is an optimal solution
		of (\ref{min:f});
		
		\item[(ii)] $\lim_{k\to\infty} f(x_k)=f^*$.
	\end{enumerate}
\end{theorem}

\bigbreak

\section{Acknowledgements} The authors are grateful to the Australian Research Council,
RMIT University  and the Australia--China YSEP program for  the financial support of research visits
that were instrumental in accomplishing this work.
HK was supported in part by NSF of China under grant number U1811461.

\bigbreak

\bibliographystyle{plain}
\bibliography{refs}

\end{document}